\documentclass{amsart}
\usepackage{amssymb,latexsym}
\usepackage{amsfonts}
\usepackage[mathscr]{eucal}
\usepackage{enumerate}

\newtheorem{thm}{Theorem}[section]

\newtheorem{cx}[thm]{Conjecture}
\theoremstyle{remark}
\newtheorem{zj}{Remark}[section]
\theoremstyle{definition}
\newtheorem{dy}{Definition}[section]
\numberwithin{equation}{section}
\allowdisplaybreaks[4]

\renewcommand{\dim}{\operatorname{dim}}

\begin{document}
\title{ Chern number inequalities of deformed Hermitian-Yang-Mills metrics on four dimensional K\"ahler manifolds}
\author{Xiaoli Han}
\address{Xiaoli Han\\ Math department of Tsinghua university\\ Beijing\\ 100084\\ China\\} \email{hanxiaoli@mail.tsinghua.edu.cn}
\author{Xishen Jin}
\address{Xishen Jin\\ Department of Mathematics\\ Renmin University of China \\ Beijing\\ 100872\\ China\\} \email{jinxishen@ruc.edu.cn}

\begin{abstract}
  In this paper, we investigate the Chern number inequalities on $4$-dimensional K\"ahler manifolds admitting the deformed Hermitian-Yang-Mills metrics under the assumption $\hat\theta\in (\pi,2\pi)$.
\end{abstract}


\maketitle

\section{Introduction}
Let $(X,\omega)$ be a compact K\"ahler manifold of complex dimension $n$ and $L$ be a holomorphic line bundle over $X$. Given a Hermitian metric $h$ on $L$, we define a complex function $\zeta$ on $X$ by
\[
\zeta:=\frac{(\omega-F_h)^n}{\omega^n}
\]
where $F_h=-\partial\bar \partial \log h$ is the curvature of the Chern connection with respect to the metric $h$. It is easy to see that the average of this function is a fixed complex number
\[
Z_{L,[\omega]}:= \int_X \zeta \frac{\omega^n}{n!}
\]
depending only on the cohomology classes $c_1(L)$ and $[\omega]\in H^{1,1}(X,\mathbb{R})$. Let $\hat \theta$ be the argument of $Z_{L,[\omega]}$.

\begin{dy}
  A Hermitian metric $h$ on $L$ is said to be a deformed Hermitian-Yang-Mills (dHYM) metric if it satisfies
  \begin{equation}
    \label{eqn-dHYM-form}
    \operatorname{\mathop{Im}}(\omega-F)^n=\tan(\hat \theta) \operatorname{\mathop{Re}}(\omega-F)^n.
  \end{equation}
\end{dy}

The equation \eqref{eqn-dHYM-form} is called dHYM equation. The dHYM equation was first discovered by physical scientists Marino et all \cite{MMMS} as the requirement for a D-brane on the B-model of mirror symmetry to be supersymmetric. This phenomenon was explained by Leung-Yau-Zaslow \cite{LYZ} in mathematical language. From a viewpoint of differential geometry, this might be thought of a relationship between the existence of `nice' metrics on the line bundle over Calabi-Yau manifolds and special Lagrangian submanifolds in another Calabi-Yau manifolds.  Recently, the dHYM metrics have been studied actively(e.g. \cite{CJY}, \cite{CXY}, \cite{CY}, \cite{HY}, \cite{HJ}, \cite{JY}, \cite{Ping}, \cite{SS}, \cite{T2019} etc). See also \cite{CollinsShi} for a survey.

We define the Lagrangian Phase operator $\theta: \wedge^{1,1}X \to \mathbb{R}$ by
\[
\theta(F)=\sum_{j=1}^n\arctan\lambda_j,
\]
where $\lambda_j$($j=1,\cdots,n$) are the eigenvalues of $\omega^{-1}F \in \mathrm{End}(T^{1,0}X)$. Then according to arguments in \cite{JY}, the equation (\ref{eqn-dHYM-form}) is equivalent to
\begin{equation}\label{eqn-dHYM-angle}
\theta(F)=\hat\theta(\text{  mod  } 2\pi).
\end{equation}
In particular, we always change the interval of argument such that $\hat \theta=\theta(F)$ if $F$ is a dHYM metric in the whole paper. This angle $\hat \theta$ is also called the analytic lifted angle.

Given two $(1,1)$-class, let us consider the ``winding angle'' of
\[
Z_{[\omega],c_1(L)}(t) =-\int_X e^{-t\sqrt{-1} \omega} ch(L)
\]
as $t$ runs from $+\infty$ to $1$ if $Z_{[\omega],c_1(L)}(t)$ does not cross $0\in \mathbb{C}$. Here
\[
ch(L)=e^{c_1(L)}
\]
is the Chern character of $L$. More precisely,
\[
Z_{[\omega],c_1(L)}(t)=-\int_X \frac{(c_1(L)-t\sqrt{-1}\omega)^n}{n!}.
\]
This ``winding angle'' is also called algebraic lifted angle. The motivation to define the algebraic lifted angle is to compute the analytic lifted angle whenever the later exists. As discussed in \cite{CXY}, in dimension less than $4$, if $F$ is a dHYM metric and $\hat \theta\in (\frac{(n-2)\pi}{2},\frac{n\pi}{2})$, then the analytic lifted angle $\hat\theta$ equals the algebraic lifted angle.

As pointed in \cite{CollinsShi}, the only obstacle to define the algebraically lifted angle is the possibility that $Z_{[\omega],c_1(L)}(t)=0$ for some $T\in [1,+\infty)$, since if this case occurs, the ``winding angle'' is no-longer well-defined. When $\dim_{\mathbb{C}} X=1$, this can not occur since $\operatorname{\mathop{Im}}(Z_{[\omega],c_1(L)}(t))>0$. When $\dim_{\mathbb{C}} X=2$, $Z_{[\omega],c_1(L)}(t)\in \mathbb{C}^*=\mathbb{C}\setminus \{0\}$ by the Hodge index theorem as in \cite{CXY}. In the case of $\dim_{\mathbb{C}}X >2$, Collins-Yau \cite{CY} gave an example model on $Bl_p\mathbb{P}^3$ such that $Z_{[\omega],c_1(L)}(t)$ can pass through the origin. Collins-Xie-Yau \cite{CXY} proved the following Chern number inequality for $\dim_{\mathbb{C}} X=3$.

\begin{thm}[\cite{CXY}]
\label{Thm-CXY}
Suppose $(X,\omega)$ is $3$-dimensional K\"ahler manifold and $L$ admits a dHYM metric with analytic lifted angle $\hat \theta\in (\frac{\pi}{2},\frac{3\pi}{2})$. Then
\[
\left(\int_X \omega^3\right)\left(\int_X c_1(L)^3\right)< 9\left(\int_X c_1(L)^2\wedge \omega\right)\left(\int_X c_1(L)\wedge\omega^2\right).
\]
\end{thm}

The Chern number inequalities play very important roles in the study of canonical metrics. For example, in the study of Hermian-Einstein metrics, Bogomolov \cite{Bogomolov} obtained the Bogomolov inequality for semi-stable holomorphic vector bundle and Simpson \cite{Simpson} proved the Bogomolov inequality for stable Higgs bundles on compact K\"ahler manifolds by constructing Higgs Hermitian-Einstein metrics. Furthermore, in the study of K\"ahler-Einstein metrics, Miyaoka \cite{Miyaoka1} and Yau \cite{Yau2} proved the famous Miyaoka-Yau inequality. The Miyaoka-Yau inequality can also be obtained by constructing Higgs structure on $T^{1,0}(X)\otimes \mathcal{O}_X$. A celebrated conjecture by Thomas-Yau \cite{Thomas2001,ThomasYau2002} is that the existence of special Lagrangian submanifolds in Calabi-Yau manifolds is equivalent to some stability conditions. In mirror symmetry, some stability conditions are needed in order to obtain the existence of dHYM metrics. This was studied in a series of works \cite{CJY,CXY,CY}. The Chern number inequality in Theorem \ref{Thm-CXY} is precisely the line bundle case of the inequality conjectured by Bayer-Macri-Toda \cite{BayerMacriToda} in their construction of Bridgeland stability conditions.
Furthermore, as described in \cite{CY}, the Chern number inequalities ensure that $Z_{[\omega],c_1(L)}(t)\in \mathbb{C}^*$ for all $t\in [1,+\infty)$. So the ``winding angle'' can be well-defined and also is the algebraic lifted angle.

In \cite{CY}, Collins-Yau proposed the following conjecture of Chern number inequalities on the $4$-dimension case.


\begin{cx}[\cite{CY}]
\label{cx-CY}
  Suppose $(X,\omega)$ is a four dimensional compact K\"ahler manifold and $L$ admits a dHYM metric $F$ with constant angle $\hat\theta\in (\frac{3\pi}{2},2\pi)$. Then the following Chern number inequalities hold
  \[
  \frac{c_1(L)^3\cdot \omega}{c_1(L)\cdot \omega^3}>1
  \]
  and
  \[
  \frac{(c_1(L)^3\cdot \omega)(\omega^4)}{c_1(L)\cdot \omega^3}-6(c_1(L)^2\cdot \omega^2) +\frac{(c_1(L)\cdot \omega^3)(c_1(L)^4)}{c_1(L)^3\cdot \omega}<0.
  \]
\end{cx}

%
%

Let us recall the motivation of Conjecture \ref{cx-CY} in \cite{CY}. If $\dim_{\mathbb{C}}X =4$, the path $Z_{[\omega],c_1(L)}(t)$ can be expressed by
\[
-\int_X(t^4\omega^4- 6t^2\omega^2\wedge c_1(L)^2 + c_1(L)^4) -4t\sqrt{-1}\int_X (t^2c_1(L)\wedge \omega^3 - c_1(L)^3\wedge \omega).
\]
For $t\approx+\infty$, $Z_{[\omega],c_1(L)}(t)$ lies near the negative real axis. Furthermore, at $t=1$,
\[
Z_{[\omega],c_1(L)}(1)=-\frac{1}{4!}\int_X(c_1(L)-\sqrt{-1}\omega)^4=-e^{\sqrt{-1}\hat\theta}R([\omega],c_1(L)),
\]
where $R([\omega],c_1(L))$ is a positive constant. Hence, if $\hat\theta\in (\pi,2\pi)$, then $Z_{[\omega],c_1(L)}(1)$ lies on the upper half-plane. Then $Z_{[\omega],c_1(L)}(t)$ must cross the positive real axis at some $t=T^*>1$, i.e. we must have $\operatorname{\mathop{Im}}(Z_{[\omega],c_1(L)}(T^*))=0$. So we should have
\[
\frac{c_1(L)^3\cdot \omega}{c_1(L)\cdot \omega^3}>1.
\]
Furthermore, at $T^*$ we must have that $\operatorname{\mathop{Re}}(Z_{[\omega],c_1(L)}(T^*))>0$ which is just the second Chern number inequality. Hence, we believe the Chern number inequalities should hold if $L$ admits a dHYM metric $F$ with constant angle $\hat\theta\in (\pi,2\pi)$.

In this paper, we prove the Chern number inequalities under the assumption $\hat\theta\in (\pi,2\pi)$ using the Khovanskii-Teissier inequalities in \cite{Collinsconcave,JianXiao}. Actually, we can prove the following theorem.

\begin{thm}
  Suppose $(X,\omega)$ is a four dimensional K\"ahler manifold and $L$ admits a dHYM metric $F$ with constant angle $\hat\theta\in (\pi,2\pi)$. Then the following Chern number inequalities hold
  \[
  \frac{c_1(L)^3\cdot \omega}{c_1(L)\cdot \omega^3}>1
  \]
  and
  \[
  \frac{(c_1(L)^3\cdot \omega)(\omega^4)}{c_1(L)\cdot \omega^3}-6(c_1(L)^2\cdot \omega^2) +\frac{(c_1(L)\cdot \omega^3)(c_1(L)^4)}{c_1(L)^3\cdot \omega}<0.
  \]
\end{thm}

{\bf Acknowledgements:} The authors would like to thank T. Collins and J. Xiao for some helpful discussions. The research is partially supported by NSFC 11721101. The second author is also supported by the Fundamental Research Funds for the Central Universities and the Research Funds of Renmin University of China.

\section{Preliminaries on Khovanskii-Teissier inequalities}
\label{sec-2}
In this section, we review the Khovanskii-Teissier inequality. We always assume that $(X,\omega)$ is an $n$-dimension compact K\"ahler manifold. The original Khovanskii-Teissier inequality was discovered by Khovanskii \cite{Khovanskii} and Teissier \cite{Teissier1,Teissier2}. The Khovanskii-Teissier inequality tells us that for any $\beta\in H^{1,1}(X,\mathbb{R})$, there holds
\[
\left(\int_X \beta^2\wedge \omega^{n-2}\right)\left(\int_X\omega^n\right)\leq \left(\int_X\beta\wedge \omega^{n-1}\right)^2.
\]
This inequality can be viewed as a generalization of the Hodge Index Theorem. In fact, the Khovanskii-Teissier inequality has been extended beyond the K\"ahler cone, see \cite{Boucksom,Collinsconcave,JianXiao}. In this paper, we need the generalized Khovanskii-Teissier inequalities related to the complex Hessian equations. These generalized Khovanskii-Teissier inequalities can be found in \cite{Collinsconcave,JianXiao}. We first recall the definition of $k$-positive cone $\mathcal{K}_{\Gamma_k,\omega}$ in $H^{1,1}(X,\mathbb{R})$ as in \cite{Collinsconcave}.

\begin{dy}[\cite{Collinsconcave}]
  We denote $\mathcal{K}_{\Gamma_k,\omega}\subset H^{1,1}(X,\mathbb{R})$ the cone of classes admitting a $k$-positive representative with respect to $\omega$, i.e., for any $[\alpha]\in \mathcal{K}_{\Gamma_k,\omega}$, there exists a closed $(1,1)$-form $\alpha\in [\alpha]$ such that the eigenvalues $\Lambda=(\lambda_1,\cdots,\lambda_n)$ satisfies
  \[
  \sigma_i(\Lambda)>0
  \]
  for $i=1,\cdots,k$, where $\sigma_i$ is the $i$-th elementary symmetric polynomial.
\end{dy}

\begin{zj}
  $\mathcal{K}_{\Gamma_k,\omega}$ is an open convex cone in $H^{1,1}(X,\mathbb{R})$. It is easy to see
  \[
  \mathcal{K}_{\Gamma_n,\omega}\subset \mathcal{K}_{\Gamma_{n-1},\omega}\subset\cdots\subset\mathcal{K}_{\Gamma_1,\omega}
  \]
  and $\mathcal{K}_{\Gamma_n,\omega}$ is the K\"ahler cone.
\end{zj}

Collins \cite{Collinsconcave} and Xiao \cite{JianXiao} proved the following generalized Khovanskii-Teissier inequalities on $\mathcal{K}_{\Gamma_k,\omega}$ by using the complex Hessian equations and the concavity inequalities for elementary symmetric polynomials.

\begin{thm}[\cite{Collinsconcave,JianXiao}]
\label{K-T}
Let $(X,\omega)$ be a compact $n$-dimensional K\"ahler manifold. For any $[\alpha]\in \mathcal{K}_{\Gamma_k,\omega}$ and $[\beta]\in H^{1,1}(X,\mathbb{R})$, there holds the following generalized Khovanskii-Teissier type inequalities
\[
\left(\omega^{n-m}\cdot\alpha^{m-2}\cdot\beta^2\right)\left(\omega^{n-m}\cdot\alpha^m\right)\leq \left(\omega^{n-m+1}\cdot\beta\cdot\alpha^{m-1}\right)^2,
\]
for all $m=2,\cdots, k$. Equalities hold if and only if $[\beta]=\lambda[\alpha]$ for some $\lambda\in \mathbb{R}$.
\end{thm}

\section{Chern number inequalities in four dimension}
\label{sec-3}
In this section, we first prove an inequality in the K\"ahler cone.

\begin{thm}
\label{thm-1}
  Suppose $\alpha$ and $\omega$ are any two K\"ahler classes on $M$, then we have the following inequality
  \[
{\frac{\left(\alpha^{3} \cdot \omega\right)\left(\omega^{4}\right)}{\alpha \cdot \omega^{3}}-2\left(\alpha^{2} \cdot \omega^{2}\right)+\frac{\left(\alpha \cdot \omega^{3}\right)\left(\alpha^{4}\right)}{\alpha^{3} \cdot \omega}\leq 0}.
  \]
  Moreover, the equality holds if and only if $\alpha$ and $\omega$ are proportional.
\end{thm}

\begin{proof}
According to Theorem \ref{K-T}, we have
\[
(\alpha^{k+1}\cdot \omega^{3-k})(\alpha^{k-1}\cdot \omega^{5-k})\leq (\alpha^k \cdot\omega^{4-k})^2,
\]
for all $k=1,2,3$. More precisely, we have
\begin{equation}
  \label{eqn-k=1}
  (\alpha^{2}\cdot \omega^{2})( \omega^{4})\leq (\alpha \cdot\omega^{3})^2,
\end{equation}
\begin{equation}
  \label{eqn-k=2}
  (\alpha^{3}\cdot \omega)(\alpha\cdot \omega^{3})\leq (\alpha^2 \cdot\omega^{2})^2,
\end{equation}
and
\begin{equation}
  \label{eqn-k=3}
  (\alpha^{4})(\alpha^{2}\cdot \omega^{2})\leq (\alpha^3 \cdot\omega)^2.
\end{equation}

Multiplying Equation \eqref{eqn-k=1} and \eqref{eqn-k=2}, we have
\begin{equation}
  \label{eqn-12}
  (\omega^4)(\alpha^{3}\cdot \omega)\leq (\alpha \cdot\omega^3)(\alpha^2 \cdot\omega^2).
\end{equation}
Also multiplying Equation \eqref{eqn-k=2} and \eqref{eqn-k=3}, we have
\begin{equation}
  \label{eqn-23}
  (\alpha\cdot\omega^3)(\alpha^4)\leq (\alpha^3 \cdot\omega)(\alpha^2 \cdot\omega^2).
\end{equation}

Since $\alpha^3\cdot \omega>0$ and $\alpha\cdot\omega^3>0$, multiplying them to the both sides of Equation \eqref{eqn-12} and \eqref{eqn-23} respectively, we obtain
\[
(\alpha^3\cdot \omega)(\omega^4)(\alpha^{3}\cdot \omega)\leq (\alpha \cdot\omega^3)(\alpha^2 \cdot\omega^2)(\alpha^3\cdot \omega)
\]
and
\[
(\alpha\cdot\omega^3)(\alpha\cdot\omega^3)(\alpha^4)\leq (\alpha\cdot\omega^3) (\alpha^3 \cdot\omega)(\alpha^2 \cdot\omega^2).
\]

Adding these two inequalities together, we get
\[
\begin{split}
&(\omega^4)(\alpha^{3}\cdot \omega)^2 +(\alpha\cdot\omega^3)^2(\alpha^4)\\
\leq &2(\alpha \cdot\omega^3)(\alpha^2 \cdot\omega^2)(\alpha^3\cdot \omega).
\end{split}
\]
Hence, we have
\[
{\frac{\left(\alpha^{3} \cdot \omega\right)\left(\omega^{4}\right)}{\alpha \cdot \omega^{3}}-2\left(\alpha^{2} \cdot \omega^{2}\right)+\frac{\left(\alpha \cdot \omega^{3}\right)\left(\alpha^{4}\right)}{\alpha^{3} \cdot \omega}<0}.\qedhere
\]

\end{proof}

%

The range $\hat\theta\in (\frac{3\pi}{2},2\pi)$ implies that $c_1(L)$ is actually a K\"ahler class. Then Conjecture \ref{cx-CY} can be easily obtained as an application of Theorem \ref{thm-1}.

\begin{proof}
We assume the Hermitian metric $h$ on $L$ is a dHYM metric and the induced Chern curvature of $h$ is denoted by $F$. We also denote $\lambda_i(i=1,\cdots,4)$ are the eigenvalues of the endormorphism $\omega^{-1}\sqrt{-1}F\in \text{End}(T^{1,0}X)$. Then the dHYM metric condition is equivalent to the following equation on $\lambda_i$:
\[
\hat \theta=\sum_{i=1}^4 \arctan \lambda_i.
\]
Since $\hat\theta\in (\frac{3\pi}{2},2\pi)$, we know that $\lambda_i>0$ and $\lambda_i\lambda_j>1$ for $i\neq j$, i.e., $c_1(L)$ lies in the K\"ahler cone. Since $\lambda_i\lambda_j>1$, we get the first inequality immediately.

Choosing $\alpha=c_1(L)$ in Theorem \ref{thm-1}, we have
\[
{\frac{\left(c_{1}(L)^{3} \cdot \omega\right)\left(\omega^{4}\right)}{c_{1}(L) \cdot \omega^{3}}-2\left(c_{1}(L)^{2} \cdot \omega^{2}\right)+\frac{\left(c_{1}(L) \cdot \omega^{3}\right)\left(c_{1}(L)^{4}\right)}{c_{1}(L)^{3} \cdot \omega}\leq 0}.
\]
Then we obtain the second inequality since $c_{1}(L)^{2} \cdot \omega^{2}$ is strictly positive.
\end{proof}

At last, we prove the Chern number inequalities while $\hat\theta\in (\pi,\frac{3\pi}{2})$.

\begin{thm}
  The Chern number inequalities also hold while $\hat\theta\in (\pi,\frac{3\pi}{2})$.
\end{thm}

\begin{proof}
  For convenience, we always normalize the volume of $(X,\omega)$ is $1$, i.e.,
  \[
  \int_X\omega^4=1.
  \]
  We assume that $\alpha=\sqrt{-1}F\in c_1(L)$ is a dHYM metric and $\lambda_1\leq \lambda_2\leq \lambda_3\leq\lambda_4$ are the eigenvalues of $\omega^{-1}\alpha$. It is well-know that all $\lambda_i$ are continuous functions. Let $\sigma_k$ be the $k$-th elementary symmetric polynomial with respect to $(\lambda_1,\cdots,\lambda_4)$. Furthermore, according to \cite{WangYuan}, $\sigma_i>0$ for $i=1,2,3$, i.e. $\alpha\in\mathcal{K}_{\Gamma_3,\omega}$.

  Case 1: $\inf\limits_X \lambda_1\geq 0$, i.e., $\alpha$ is semi-positive, then we get the needed inequalities by Theorem \ref{thm-1}. Indeed, the Chern number inequalities hold for any $\alpha+\varepsilon \omega$ ,$\varepsilon> 0$. Let $\varepsilon\to 0$, we get the Chern number inequalities.

  Case 2: $\inf\limits_X \lambda_1 < 0$. For convenience, we assume that $\lambda_1(x_0)< 0$ at $x_0\in X$. According to the dHYM equation
  \[
  \sum_{i=1}^4\arctan \lambda_i=\hat \theta
  \]
  and $\hat\theta\in (\pi,\frac{3\pi}{2})$, we know that $\lambda_4\geq \lambda_3\geq \lambda_2>|\lambda_1|\geq 0$ and $\lambda_3\lambda_4\geq \lambda_2\lambda_4>1$. Then at $x_0$,
  \begin{equation}
    \begin{split}
      &\sigma_2-1-\sigma_4\\
      = &\sum_{i\neq j}\lambda_i\lambda_j -1-\lambda_1\lambda_2\lambda_3\lambda_4\\
      > & \lambda_1(\lambda_2+\lambda_3+\lambda_4)+\lambda_2\lambda_3+\lambda_3\lambda_4-\lambda_1\lambda_2\lambda_3\lambda_4\\
      \geq &\lambda_1(\lambda_2+\lambda_4)+\lambda_2\lambda_3+\lambda_3\lambda_4\\
      \geq & (\lambda_1+\lambda_4)(\lambda_2+\lambda_4)>0.
    \end{split}
  \end{equation}
 Since $\tan\hat\theta=\frac{\sigma_3-\sigma_1}{\sigma_2-1-\sigma_4}>0$,  then $\sigma_3>\sigma_1$ at $x_0$. Furthermore, $\sigma_3-\sigma_1>0$  at all $x\in X$.  Furthermore, at all $x\in X$, there holds
  \[
  \begin{split}
    &\sigma_3-\sigma_1\sigma_2+\sigma_1\\
    \leq&\sigma_3-\sigma_1(\sigma_2-\lambda_2\lambda_4)\\
    =&-\lambda_1^2\lambda_2-\lambda_1^2\lambda_3-\lambda_1^2\lambda_4-\lambda_1\lambda_2^2 -\lambda_1\lambda_2\lambda_3-\lambda_2^2\lambda_3 -\lambda_2\lambda_3\lambda_4 -\lambda_1\lambda_2\lambda_3 -\lambda_1\lambda_3^2\\
    & -\lambda_1\lambda_3\lambda_4-\lambda_2\lambda_3^2 -\lambda_3^2\lambda_4 -\lambda_1\lambda_2\lambda_4-\lambda_1\lambda_3\lambda_4-\lambda_1\lambda_4^2 -\lambda_3\lambda_4^2\\
    =&-(\lambda_2+\lambda_3+\lambda_4)(\lambda_1+\lambda_3)(\lambda_1+\lambda_2) -\lambda_3\lambda_4(\lambda_1+\lambda_3)-\lambda_4^2(\lambda_1+\lambda_3)<0.
  \end{split}
  \]
  Hence we have $\sigma_1\sigma_2>\sigma_1+\sigma_3>2\sigma_1$ and $\sigma_2>2$ at all $x\in X$.

  In conclusion, we have the following inequalities pointwise on $X$
  \begin{equation}
    \label{eqn:sigma_13}
    \sigma_3-\sigma_1>0,
  \end{equation}
    \begin{equation}
    \label{eqn:sigma_024}
    \sigma_2-\sigma_4-1>0,
  \end{equation}
  and
  \begin{equation}
  \label{sigma2}
    \sigma_2>2.
  \end{equation}
  Then we claim that
  \[
  \int_X\sigma_3-\int_X\sigma_1\int_X\sigma_2+\int_X\sigma_1<0,
  \]
  where
  \begin{equation}\label{eqn1}
  \int_X\sigma_1= 4\int_X \omega^3\wedge \alpha,
  \end{equation}
  \begin{equation}\label{eqn2}
  \int_X \sigma_2=6\int_X \omega^2\wedge \alpha^2,
  \end{equation}
  and
  \begin{equation}\label{eqn3}
  \int_X \sigma_3=4\int_X\omega\wedge\alpha^3,
  \end{equation}
  according to the definition of elementary symmetric polynomials. Indeed, since $\alpha\in \mathcal{K}_{\Gamma_3,\omega}$, by Theorem \ref{K-T} we have,
  \[
  0<\int_X \omega^3\wedge \alpha\int_X\omega\wedge \alpha^3\leq (\int_X\omega^2\wedge\alpha^2)^2,
  \]
  and
  \[
  0<\int_X \omega^2\wedge \alpha^2\int_X \omega^4\leq (\int_X\omega^3\wedge \alpha)^2.
  \]
  Multiplying these two inequalities together and applying the Equations \eqref{eqn1}-\eqref{eqn3}, we get the following inequality
  \[
  \int_X \sigma_3\leq \frac{1}{6}\int_X\sigma_1 \int_X\sigma_2.
  \]
  Then,
  \[
  \begin{split}
  &\int_X\sigma_3-\int_X\sigma_1\int_X\sigma_2+\int_X\sigma_1\\
  \leq &-\frac{5}{6}\int_X\sigma_1\int_X\sigma_2+\int_X\sigma_1\\
  <&-\frac{5}{3}\int_X\sigma_1+\int_X\sigma_1\\
  =&-\frac{2}{3}\int_X \sigma_1<0.
  \end{split}
  \]
  Since $\tan\hat \theta=\frac{\int_X \sigma_3-\int_X \sigma_1}{\int_X \sigma_2-1-\int_X\sigma_4}>0$, we know that
  \begin{equation}
  \label{eqn-need1}
  \frac{\int_X \sigma_3-\int_X \sigma_1}{\int_X \sigma_2-1-\int_X\sigma_4}(\int_X\sigma_3-\int_X\sigma_1\int_X\sigma_2+\int_X\sigma_1)<0<(\int_X\sigma_1)^2.
  \end{equation}
  Multiplying $\int_X \sigma_2-1-\int_X\sigma_4>0$ on both sides of the inequality \eqref{eqn-need1}, we get
  \[
  (\int_X\sigma_3)^2 -\int_X\sigma_1\int_X\sigma_2\int_X\sigma_3+(\int_X\sigma_1)^2\int_X\sigma_4<0
  \]
  which is equivalent to
  \[
  {\frac{\left(c_{1}(L)^{3} \cdot \omega\right)\left(\omega^{4}\right)}{c_{1}(L) \cdot \omega^{3}}-6\left(c_{1}(L)^{2} \cdot \omega^{2}\right)+\frac{\left(c_{1}(L) \cdot \omega^{3}\right)\left(c_{1}(L)^{4}\right)}{c_{1}(L)^{3} \cdot \omega}\leq 0}
  \]
  according to Equations \eqref{eqn1}-\eqref{eqn3}.\qedhere
\end{proof}

\begin{zj}
  The key estimates in the proof above are the inequalities \eqref{eqn:sigma_13}-\eqref{sigma2}. While proving these inequalities, we do not distinguish the sign of $\lambda_1$. We can also prove $\sigma_2>2$ in the following way.

  At the point $x\in X$ where $\lambda_1(x)\leq 0$, we have
  \[
  \arctan\lambda_1 +\arctan\lambda_2+\arctan\lambda_3\in(\frac{\pi}{2},\pi)
  \]
  according to $\hat\theta\in (\pi,\frac{3\pi}{2})$. Since
  \[
  0>\tan(\arctan\lambda_1 +\arctan\lambda_2+\arctan\lambda_3)=\frac{\lambda_1+\lambda_2+\lambda_3-\lambda_1\lambda_2\lambda_3}{1-\lambda_1\lambda_2 -\lambda_2\lambda_3-\lambda_1\lambda_3}
  \]
  and $\lambda_1+\lambda_2+\lambda_3-\lambda_1\lambda_2\lambda_3>0$, we have
  \[
  \lambda_1\lambda_2 +\lambda_2\lambda_3+\lambda_1\lambda_3>1.
  \]
  Hence,
  \[
  \begin{split}
  &\sigma_2-\sigma_4-1\\
  =&\lambda_1\lambda_2 +\lambda_2\lambda_3+\lambda_1\lambda_3 +\lambda_1\lambda_4+\lambda_3\lambda_4+\lambda_2\lambda_4-\lambda_1\lambda_2\lambda_3\lambda_4-1\\
  > &(\lambda_1+\lambda_3+\lambda_2)\lambda_4-\lambda_1\lambda_2\lambda_3\lambda_4> 0.
  \end{split}
  \]
  Then we get $\sigma_3-\sigma_1>0$, since $\hat\theta\in (\pi,\frac{3\pi}{2})$ and $\tan\hat\theta=\frac{\sigma_3- \sigma_1}{\sigma_2-1-\sigma_4}>0$. Furthermore, since
  \[\arctan\lambda_2+\arctan\lambda_4\in (\frac{\pi}{2},\pi),
  \]
  and
  \[\arctan\lambda_3+\arctan\lambda_4\in (\frac{\pi}{2},\pi),
  \] we know that $\lambda_2\lambda_4>1$ and $\lambda_3\lambda_4>1$.
  Thus
  \[
  \begin{split}
\sigma_2=&\lambda_1\lambda_2 +\lambda_2\lambda_3+\lambda_1\lambda_3 +\lambda_1\lambda_4+\lambda_3\lambda_4+\lambda_2\lambda_4\\
  > &1+(\lambda_1+\lambda_3)\lambda_4+\lambda_2\lambda_4> 2.
  \end{split}
  \]

  At the points $x\in X$ where $\lambda_1(x)>0$, we have
  \[
  \sigma_2>\lambda_2\lambda_4+\lambda_3\lambda_4>2.
  \]

\end{zj}

%

\end{document}